\documentclass{amsart}

\newtheorem{theorem}{Theorem}[section]
\newtheorem{lemma}[theorem]{Lemma}
\newtheorem{prop}[theorem]{Proposition}
\theoremstyle{definition}
\newtheorem{definition}[theorem]{Definition}

\newtheorem{remark}[theorem]{Remark}
\usepackage{xcolor}

\newcommand{\be}{\begin{equation}}
\newcommand{\ee}{\end{equation}}
\newcommand{\lt}{\left}
\newcommand{\rt}{\right}

\newcommand{\al}{\alpha}

\newcommand{\ubar}{\underline}

\def\dfrac{\displaystyle\frac}

\newcommand{\td}{\tilde}
\newcommand{\p}{\partial}

\newcommand{\R}{\mathbb{R}}
\newcommand{\M}{\mathcal{M}}
\newcommand{\K}{\mathcal{K}}

\newcommand{\us}{u^*}

\newcommand{\lu}{\ubar{u}}
\newcommand{\uu}{\bar{u}}
\newcommand{\ga}{\gamma}
\newcommand{\gas}{\gamma^*}

\newcommand{\T}{\partial}

\newcommand{\w}{w^*}
\newcommand{\s}{\sigma}
\newcommand{\ka}{\kappa}
\newcommand{\goto}{\rightarrow}

\newcommand{\bn}{\bar{\nabla}}

\newcommand{\tbus}{\tilde{\bar{u}}^*}

\numberwithin{equation}{section}

\begin{document}
	
	\title[Translating solutions in Minkowski space]{Translating solutions and the entire Hessian curvature flow in Minkowski space }
	
	%    Remove any unused author tags.
	
	%    author one information
	\author{Changzheng Qu}
	\address{School of Mathematics and Statistics, Ningbo University, Ningbo, China}
	\curraddr{}
	\email{quchangzheng@nbu.edu.cn}
	\thanks{}
	
	%    author two information
	\author{Zhizhang Wang}
	\address{School of Mathematical Sciences, Fudan University, Shanghai, China.}
	\curraddr{}
	\email{zzwang@fudan.edu.cn}
	\thanks{}
	
	\author{Weifeng Wo}
	\address{School of Mathematics and Statistics, Ningbo University, Ningbo, China}
	\curraddr{}
	\email{woweifeng@nbu.edu.cn}
	\thanks{}

	\thanks{The first and third authors are supported by NSFC Grant No. 12431008. The second author is supported by NSFC Grants No.12141105. }
	
	%	\subjclass[2010]{Primary }
	
	\begin{abstract}
In this paper, we study the $k$-Hessian curvature flow of noncompact spacelike hypersurfaces in Minkowski space. We first prove the existence of translating solutions with given asymptotic behavior. Then, we prove that for strictly convex initial hypersurface satisfying certain conditions, the curvature flow exists for all time, and the normalized flow converges to a translating solution.
\end{abstract}


	\dedicatory{}
	
	\maketitle
\section{Introduction}

Let $\R^{n, 1}$ be the Minkowski space with the Lorentzian metric
\[ds^2=\sum\limits_{i=1}^ndx_i^2-dx_{n+1}^2.\]
We  denote the corresponding inner product by $\langle \cdot,\cdot\rangle$.
A spacelike hypersurface $\M\subset\R^{n, 1}$ always has an everywhere timelike normal field, which we assume to be future directed and to satisfy the condition
$\lt<\nu, \nu\rt> =-1.$ Such hypersurfaces can be locally expressed as the graph of a function $u(x): \R^n\rightarrow\R$ satisfying $|Du(x)|< 1$ for all $x\in\R^n.$ Thus, the position vector of $\M$ should be
$X(x)=(x, u(x))$ for any $x\in \mathbb{R}^n$.

Inspired by the construction of the constant curvature hypersurfaces \cite{WX20},  in the papers \cite{WX22-1,WX22-2}, Wang and Xiao studied the Hessian curvature self-expanders and the Hessian curvature flows of entire spacelike hypersurfaces in Minkowski space. Specifically, for any given positive function $\varphi$ defined on the unit sphere, they proved in \cite{WX22-1} that it is always possible to find some entire spacelike hypersurface $\M$ that satisfies the self-expander equation
\begin{eqnarray}
\sigma_k^{1/k}(\kappa[\M])=-\lambda \langle X, \nu\rangle,  \text{ for some constant } \lambda>0
\end{eqnarray}
with the asymptotic behavior
\begin{eqnarray}
u(x)\rightarrow |x|+\varphi\left(\frac{x}{|x|}\right), \text{ as } |x|\rightarrow +\infty.
\end{eqnarray}
Here $\sigma_k$ is the $k$-th elementary symmetric polynomial, i.e.,
\[\sigma_k(\kappa)=\sum\limits_{1\leq i_1<\cdots<i_k\leq n}\kappa_{i_1}\cdots\kappa_{i_k},\]
where $\kappa_i$ is the $i$th principle curvature of $\M$.

Later, in \cite{WX22-2}, they proved that for any given positive function $\varphi$ defined on the unit sphere, if the initial spacelike entire strictly convex  hypersurface $\M_0$ satisfies the condition
$$0 < \sigma_k^{1/k}(\kappa[\M_0]) < C\langle X,\nu\rangle$$
for some large positive constant $C$, then there exists a family of spacelike hypersurfaces that form a solution to the $\sigma_k^{1/k}$
curvature flow. Specifically, for each $t > 0$, $X(\mathbb{R}^n, t) = \M_t$
is an entire spacelike hypersurface embedded in $\mathbb{R}^{n,1}$
, and $X(\cdot, t)$ satisfies
\begin{eqnarray}\label{flow}
	\left\{
	\begin{aligned}
		 	\frac{\partial X(p,t)}{\partial t}&=\left(\sigma_k(\kappa[\M_t])(p,t)\right)^{1/k}\nu(p,t), \\
		 X(\cdot, 0)&=X_0
	\end{aligned}
	\right.
\end{eqnarray}
where
 \[X_0: \R^n\goto\R^{n, 1}\] is an embedding with $X_0(\R^n)=\M_0$ and  $\nu(\cdot, t)$ is the unit upward normal of $\M_t$.

Since each $\M_t$ is graphic,  we can rewrite the parabolic equation \eqref{flow} using $X(p,t)=(p,u(p,t))$ as follows:
\begin{eqnarray}\label{flow1}
	\left\{
	\begin{aligned}
			\frac{\partial u(x,t)}{\partial t}&=\sqrt{1-|Du|^2}\left(\sigma_k(\kappa[\M_t])\right)^{1/k}, \\
		u(\cdot, 0)&=u_0,
	\end{aligned}
	\right.
\end{eqnarray}
where $X_0(x)=(x, u_0(x))$. For our convenient, we slightly modify the above flow:
\begin{eqnarray}\label{ge}
	\left\{
	\begin{aligned}
			\frac{\partial u(x,t)}{\partial t}&=\sqrt{1-|Du|^2}\left(\frac{\sigma_k(\kappa[\M_t])}{C_n^k}\right)^{1/k}, \\
		u(\cdot, 0)&=u_0,
	\end{aligned}
	\right.
\end{eqnarray}
where $C_n^k$ is the binomial coefficient.

The curvature flow problem in Euclidean space has been extensively studied in the literature. For the mean curvature flow, Huisken \cite{Hui} proved that if the initial hypersurface $\M_0$ is smooth, closed, and strictly convex, then the flow exists on a finite time interval $0\leq t\leq T,$ and the $\M_t$ converges to a point as $t\goto T.$ Moreover, with  suitable rescaling, it has been shown that the normalized hypersurfaces converge to a sphere. A similar result was obtained by Chow \cite{Chow} for the $n$-th root of the Gauss curvature flow. In \cite{And}, Andrews generalized Huisken's result via the Gauss map to a large family of curvature flows including the $k$-th root of the $\sigma_k$ curvature flow. For entire graphical flows,  Daskalopoulos-Huisken \cite{DH} studied the inverse mean curvature flow and  Choi-Daskalopoulos-Kim-Lee \cite{CDKL} studied the Gauss curvature flow.

In Minkowski space,  Andrews, Chen, Fang and McCoy \cite{ACFMc} proved a similar result as \cite{And} under the assumption that the initial hypersurface $\M_0$ is spacelike, co-compact, and strictly convex.  Aarons \cite{Aa} studied the mean curvature flow with a forcing term in Minkowski space. Bayard-Schn\"{u}rer \cite{BS09} and Bayard \cite{B09} studied the Gauss curvature flow and scalar curvature flow with a forcing term in Minkowski space. 

Translating solution arises in the study of singularities of curvature flows. These solutions not only provide insight into the structure of singularities but also appear as blow-up limits in various flow problems. Huisken and Sinestrari \cite{HS1,HS2} proved that convex translating solitons in mean curvature flow within Euclidean space are associated with type II singularities. Altschuler and Wu \cite{AW94} analyzed translating solutions for the non-parametric mean curvature flow with prescribed contact angles.
It was shown by Wang \cite{Wang} that for  $n=2$, convex translating solutions must be rotationally symmetric, while for  $n\geq3$, he constructed non-rotationally symmetric, complete convex translating solitons.

Complete noncompact  singularities of curvature flows in Euclidean space have been studied extensively (\cite{ CDL21, CSS07, HIMW19, SW10, SX20, Urb98}). However, few results are known in Minkowski space.For mean curvature flow in Minkowski space, Ecker \cite{Eck} constructed a radially symmetric solution to equation (1.3), which was later analyzed in detail by Jian \cite{Jian}. Spruck and Xiao \cite{SX16} explored entire downward translating solitons, while Bayard \cite{Bay} extended the study of scalar curvature flow by examining entire downward solitons. Additionally, Ding \cite{Ding} investigated entire smooth convex strictly spacelike translating solitons for mean curvature flow in Minkowski space.

This leads to a natural question: apart from self-expanders, can other singularities be identified as limits of the Hessian curvature flow in Minkowski space? In this paper, we explore translating solutions and the corresponding deformation procedure from a strictly convex hypersurface to a given translator via the Hessian curvature flow.

A translating solution is of the form $$u(x,t)=u(x)+at,$$ where $a>0$ is the velocity. Substituting this form  into \eqref{ge}, we obtain the equation for translating solutions:
\begin{equation}\label{teq}
	\frac
	{a}{\sqrt{1-|Du(x)|^2}}=\left(\frac{\sigma_k(\kappa[\M])}{C_n^k}\right)^{1/k},
\end{equation}
Denote $E=\frac{\partial}{\partial x_{n+1}}$. We now provide an explicit definition of translation solution for $k$- Hessian curvature flow.
\begin{definition}
A $k$-convex hypersurface $\M_u$ is called a translating solution of the flow \eqref{flow}, if it satisfies the equation
$$\sigma_k^{\frac{1}{k}}(\kappa[\M_u])=-\lambda \lt<\nu_u, E\rt>, \text{ for some constant }\lambda>0.$$
Here, $X=(x,u(x))$ is the position vector of $\M_u$ and $\nu_u$ is the future timelike unit normal of $\M_u$.\end{definition}

The first main result of this paper is

\begin{theorem}
\label{es-th1}
Suppose $\varphi$ is a $C^2$ function defined on the unit sphere.
Then there exists a unique, strictly convex solution $u:\mathbb{R}^n\rightarrow\mathbb{R}$ to \eqref{teq} such that as $|x|\goto\infty,$ we have
\begin{eqnarray}\label{asymptotic behavior}
u(x)\goto |x|+\varphi\left(\frac{x}{|x|}\right).
\end{eqnarray}
Moreover, the solution exhibits the following asymptotic behavior:
\begin{eqnarray}\label{asymptotic behavior}
u(x)= |x|+\varphi\left(\frac{x}{|x|}\right)+O\left(|x|^{\frac{2n-k^2-k}{k}}e^{-\frac{2n}{k^2}|x|^k}\right),
\end{eqnarray}
as $|x|\rightarrow +\infty$. 
\end{theorem}

\begin{remark}
According to \cite{ChenQiu}, there are no  complete spacelike mean curvature translating solitons in Minkowski space. Thus, we believe that all translating solutions constructed here are incomplete. This incompleteness may be attributed to the asymptotic behavior of $u$.
\end{remark}

Similar to \cite{WX22-2}, we can establish the following theorem:

\begin{theorem}\label{theo1}
Suppose $\varphi\in C^2(\mathbb{S}^{n-1}),$ and the initial entire spacelike strictly convex hypersurface $\M_0$ satisfying
\begin{eqnarray}\label{assume}
0<\sigma_k(\kappa[\M_0])\leq -C\langle \nu, E\rangle,
\end{eqnarray}
for some large positive constant $C$, and $u_0(x)-|x|\rightarrow \varphi(x/|x|)$ as $|x|\rightarrow \infty$. Then, for any $0<a\leq C(C_n^k)^{-1/k}$, the curvature flow \eqref{ge} admits a solution $\M_{u(x,t)}$ for all $t>0.$ Moreover, the normalized flow
$$\td{X}=\lt(x, u(x, t)-at\rt)$$
converges to a translating solution $\M_{u^{\infty}}$ with the asymptotic behavior
 $$u^{\infty}(x)-|x|\goto\varphi\lt(\frac{x}{|x|}\rt),\text{ as }\ \ |x|\goto\infty.$$
\end{theorem}

The organization of this paper is as follows: In Section 2, we discuss the rotation symmetric translation solutions. Section 3 presents the translation solutions with prescribed asymptotic behavior. In section 4, we study the long time existence and convergence of the $k$-Hessian curvature flow.

\section{Radial translating solutions}
In this section, we establish the existence of an entire radial solution to (\ref{teq}), which is unique up to a constant. By appropriate scaling, we may assume that $a=1$.

\begin{theorem}\label{thm1}
	There exists a strictly radial solution  $u(|x|): \mathbb{R}^n\rightarrow \mathbb{R}$ of
	\eqref{teq}.
	The asymptotic expansion of $u(|x|)$ is
	\begin{eqnarray}\label{asy}
	u(|x|)=|x|+c_0+O\left(|x|^{\frac{2n-k^2-k}{k}}e^{-\frac{2n}{k^2}|x|^k}\right), \text{ as } |x|\rightarrow +\infty
	\end{eqnarray}
for some constant $c_0$. The radial symmetric solution is unique up to  a constant.
\end{theorem}

For radial symmetric solution, we set $u(x)=u(|x|)=u(r)$ and define $y(r)=\frac{\partial u}{\partial r}$. A straightforward calculation yields
$$D_iu=y\frac{x_i}{|x|}, D^2_{ij}u=\frac{y}{|x|}\left(\delta_{ij}-\frac{x_ix_j}{|x|^2}\right)+y'\frac{x_ix_j}{|x|^2}.$$
Therefore, the curvature reads
$$\ka[\M_u]=\frac{1}{\sqrt{1-y^2}}\left(\frac{y'}{1-y^2},\frac{y}{r},\cdots,\frac{y}{r}\right).$$
Then
	\eqref{teq} can be rewritten as the following ODE:
\begin{equation}
	\frac{1}{\left(1-y^{2}\right)^{k / 2}} \frac{y^{k-1}}{r^{k-1}}\left(\frac{k}{n} \frac{y^{\prime}}{1-y^{2}}+\frac{n-k}{n} \frac{y}{r}\right)=\left(\frac{1}{\sqrt{1-y^{2}}}\right)^{k},
\end{equation}
which simplifies to
\begin{equation}
	\frac{y^{k-1}}{r^{k-1}}\left(\frac{k}{n} \frac{y^{\prime}}{1-y^{2}}+\frac{n-k}{n} \frac{y}{r}\right)=1.
\end{equation}
This can be further rewritten as
\begin{equation}\label{yeq}
	\frac{dy}{dr}=\frac{n}{k}\left(\frac{r^{k-1}}{y^{k-1}}-\frac{n-k}{n}\frac{y}{r}\right)(1-y^{2}).
\end{equation}
Letting $z(r)=y^k(r)$, we obtain
\begin{equation}\label{zeq}
	\frac{dz(r)}{dr}=\left(nr^{k-1}-(n-k)\frac{z}{r}\right)(1-z^{\frac{2}{k}}).
\end{equation}

	We first prove the existence part.
	
\begin{prop}\label{prop1}
	There exists a solution $z(r)$ to (\ref{zeq}) such that $z(0)=0$ and
\begin{equation}\label{z2}
	0<z(r)<1, z'(r)>0\ \ \ on \ (0,\infty).
\end{equation}
Moreover, we have  	\[\lim\limits_{r\rightarrow +\infty}z(r)=1,\ and\ \lim\limits_{r\rightarrow +\infty}z'(r)=0.
	\]
\end{prop}
	
	\begin{proof}
		
	Since (\ref{zeq}) is singular at $r=0$ for $k<n$, we consider an approximation problem
	\begin{equation}\label{aeq}
		\frac{dz(r)}{dr}=\left(n(r+\epsilon)^{k-1}-(n-k)\frac{z}{(r+\epsilon)}\right)(1-z^{\frac{2}{k}})
	\end{equation}
	and
	\begin{equation}\label{icz}
		z(0)=\epsilon^k,
	\end{equation}
	where $\epsilon$ is a small constant.
	  	
	By the short time existence theorems of ODE, we get that for any $0<\epsilon<1$, there is a unique solution to \eqref{aeq} and \eqref{icz} near $r=0$. Denote the solution by $z_\epsilon(r)$, and suppose its maximum definition interval is $[0, l)$ .

 Since $z_\epsilon$ satisfies \eqref{aeq} and \eqref{icz}, we have
\[z_\epsilon'(0)=k\epsilon^{k-1}(1-\epsilon^2)>0.\]
Therefore,  $z_\epsilon(r)>0$ , for $r\in (0,r_0)$, where  $r_0<l$ is very small.

We claim that
\begin{equation}\label{z0}
z_\epsilon'(r)> 0, \text{ for any } 0<r<l.
\end{equation}
Suppose our claim is not true, and let  $0<r_1<l$ be the first value  such that  $z_\epsilon'(r_1)=0$.	
	By $z'_\epsilon(r_1)=0$, we have either $1-z_\epsilon^{\frac{2}{k}}(r_1)=0$ or $$n(r_1+\epsilon)^{k-1}-(n-k)\frac{z_\epsilon(r_1)}{r_1+\epsilon}=0.$$
	 We will consider these two cases separately.
	
	\textbf{ Case 1}: Since  $z'_{\epsilon}> 0$ for $0<r<r_1$, it follows that  $z_{\epsilon}>0$ for $0<r<r_1$. Therefore, we have $z_\epsilon(r_1)=1$ and $\epsilon^k<z_{\epsilon}(r)<1$ for $0<r<r_1$, which implies $1-z^{2/k}_{\epsilon}> 0$. Hence, for $0<r<r_1$, since $z'_{\epsilon}>0$, we have  \[n(r+\epsilon)^{k-1}-(n-k)\frac{z_\epsilon(r)}{r+\epsilon}>0.\]	  However, from \eqref{aeq}, we obtain for $0<r<r_1$,
	\begin{equation*}
	\frac{dz_{\epsilon}}{dr}\leq\left(n(r+\epsilon)^{k-1}
	-(n-k)\frac{z_{\epsilon}}{(r+\epsilon)}\right)(1-z_{\epsilon})(1+z_{\epsilon}),
\end{equation*}
 which implies
\begin{equation*}
\frac{d}{dr}(-\ln (1-z_{\epsilon}))\leq2\left(n(r+\epsilon)^{k-1}-(n-k)\frac{z_{\epsilon}}{r+\epsilon}\right).
\end{equation*}
Integrating this inequality from $0$ to $r_1$, we get
\[\ln(1-\epsilon^k)-\ln(1-z_{\epsilon}(r_1))\leq 2n\int_0^{r_1}(r+\epsilon)^{k-1}dr <+\infty.\]
Thus $z_{\epsilon}(r_1)=1$ is impossible. In the following, we can always assume $0<z_{\epsilon}<1$ for $0\leq r\leq r_1$.

		\textbf{ Case 2}: Assume that $n(r_1+\epsilon)^{k-1}-(n-k)\frac{z_\epsilon(r_1)}{r_1+\epsilon}=0$. Therefore, for $0<r<r_1$, since $z'_{\epsilon}>0$, we also have  $n(r+\epsilon)^{k-1}-(n-k)\frac{z_\epsilon(r)}{r+\epsilon}>0$.	   However, by \eqref{aeq}, we obtain for $0<r<r_1$,
	\begin{eqnarray*}
	\frac{d z_{\epsilon}}{dr}<n(r+\epsilon)^{k-1}-(n-k)\frac{z_{\epsilon}}{(r+\epsilon)},
	\end{eqnarray*}
	namely,
 	\begin{eqnarray*}	
 		\frac{d}{dr}((r+\epsilon)^{n-k}z_{\epsilon})<n(r+\epsilon)^{n-1}.
 		\end{eqnarray*}
 	Integrating it from $0$ to $r_1$ yields
 	\begin{equation}
 		(r_1+\epsilon)^{n-k}z(r_1)-\epsilon^n<(r_1+\epsilon)^n-\epsilon^n.
 	\end{equation}
 	Thus we get \[z_\epsilon(r_1)<(r_1+\epsilon)^k,\]
 which implies
	$$n(r_1+\epsilon)^{k-1}-(n-k)\frac{z_\epsilon(r_1)}{r_1+\epsilon}>0.$$
	This results in a contradiction.
	
	Thus, we have proved that $z'_{\epsilon}(r)>0$ for $0<r<l$.  Furthermore, using a similar argument as above, we can prove that for $0<r<l$,
	\begin{equation}\label{z1}
	0<z_{\epsilon}(r)<1 \text{ and } z_{\epsilon}(r)<(r+\epsilon)^k.
	\end{equation}
	Therefore by the ODE theory, we conclude that $l=+\infty$.
	
	Let's derive a more explicit estimate for $1-z_{\epsilon}$.  It is clear that
	$$1-z_{\epsilon}=\left(1-z^{\frac{1}{k}}_{\epsilon}\right)\left(1+z_{\epsilon}^{\frac{1}{k}}+\cdots+z^{\frac{k-1}{k}}_{\epsilon}\right),$$ which implies
	$$\left(1-z^{\frac{1}{k}}_{\epsilon}\right)\leq 1-z_{\epsilon}\leq n\left(1-z^{\frac{1}{k}}_{\epsilon}\right).$$
	Therefore, from \eqref{aeq}, we have
	\begin{equation*}
	\frac{1}{n}\left(n(r+\epsilon)^{k-1}-(n-k)\frac{z_{\epsilon}}{(r+\epsilon)}\right)(1-z_{\epsilon})\leq \frac{dz_{\epsilon}}{dr}\leq 2n(r+\epsilon)^{k-1}(1-z_{\epsilon}).
\end{equation*}	
Since $z_{\epsilon}<(r+\epsilon)^k$, the above equation implies
\begin{equation*}
	\frac{k}{n}(r+\epsilon)^{k-1}(1-z_{\epsilon})\leq \frac{dz_{\epsilon}}{dr}\leq 2n(r+\epsilon)^{k-1}(1-z_{\epsilon}).
\end{equation*}	
Thus, integrating from $0$ to $r$, we get
\begin{equation}\label{1-z}
\exp\left(-\frac{2n}{k}\left((r+\epsilon)^k-\epsilon^k\right)\right)\leq \frac{1-z_{\epsilon}(r)}{1-\epsilon^k}\leq \exp\left(-\frac{1}{n}\left((r+\epsilon)^k-\epsilon^k\right)\right).
\end{equation}
Thus, by \eqref{aeq}, we obtain that
$$0<\frac{d z_{\epsilon}}{dr}<C,$$
where $C$ is a uniform constant.

Taking the  derivative of \eqref{aeq}, we have
  \begin{eqnarray}
		z''&=&\left(n(k-1)(r+\epsilon)^{k-2}+(n-k)\frac{z}{(r+\epsilon)^2}-(n-k)\frac{z'}{r+\epsilon}\right)(1-z^{\frac{2}{k}})\nonumber\\
		&&+\left(n(r+\epsilon)^{k-1}-(n-k)\frac{z}{(r+\epsilon)}\right)\left(-\frac{2}{k}z^{\frac{2}{k}-1}z'\right)\nonumber.
	\end{eqnarray}
Using \eqref{aeq} and \eqref{1-z}, we get the uniform bound of $z''_{\epsilon}(r)$, i.e.
	$$\left|\frac{d^2 z_{\epsilon}}{dr^2}\right|<C,$$
where $C$ is a uniform constant.

	Since $z_\epsilon$ has uniform $C^0$ and $C^2$ bounds on $[0,\infty)$,   we can let $\epsilon\rightarrow 0$ and obtain a solution $z$ of \eqref{zeq}. Moreover, by \eqref{z1}, we have $z(0)=0$, it follows that  $$0\leq z\leq 1 \text{and} \ z\leq r^k.$$ Therefore, $z'\geq 0$.
Furthermore, by applying a similar argument as before, we can demonstrate that $z'>0$ and $0<z<1$ for $r>0$.   Letting $\epsilon\rightarrow 0$, we derive
 \begin{equation}\label{1-z*}
\exp\left(-\frac{2n}{k}r^k\right)\leq 1-z(r)\leq \exp\left(-\frac{1}{n}r^k\right),
\end{equation}
which implies $$\lim_{r\rightarrow +\infty} z=1\  \text{and}\  \lim_{r\rightarrow+\infty}z'=0$$ by \eqref{zeq}.

	This completes the proof.

	\end{proof}

	\begin{prop}\label{prop2}
	Let $y(r)$ be the solution of (\ref{yeq}) as constructed in Proposition \ref{prop1}. Then as $r\rightarrow \infty$,  $y(r)$ has the following asymptotic expansion:
	\[y(r)=1-Cr^{\frac{2(n-k)}{k}}e^{-\frac{2n}{k^2}r^k}+o(r^{\frac{2(n-k)}{k}}e^{-\frac{2n}{k^2}r^k}),\]
	where $C$ is a positive constant independent of  $r$.
		\end{prop}
	
\begin{proof}
Suppose $y(r)$ is the solution constructed in  Proposition \ref{prop1}.	We define the function $C(r)$ as follows:
	\begin{equation}\label{z}
 y(r)=1-C(r)r^{\frac{2(n-k)}{k}}e^{-\frac{2n}{k^2}r^k}.\end{equation}
Since $0\leq y<1$, it follows that $C(r) > 0$. Rewriting \eqref{yeq}, we have:
	\begin{equation}\label{ode3}
		\frac{k}{n}(\ln(1-y))'=\left(\frac{n-k}{n}\frac{y}{r}-\frac{r^{k-1}}{y^{k-1}}\right)(1+y).
	\end{equation}
 Using \eqref{z}, we obtain:
	\[\ln(1-y)=-\frac{2n}{k^2}r^k+2\frac{n-k}{k}\ln r+\ln C(r),\]
and	\[\frac{k}{n}(\ln(1-y))'=-2r^{k-1}+2\frac{n-k}{n}\frac{1}{r}+\frac{k}{n}\frac{C'(r)}{C(r)}.\]
We further have
	\begin{eqnarray*}
		&&\frac{n-k}{n}\frac{y}{r}-\frac{r^{k-1}}{y^{k-1}}\\
		=&&\frac{n-k}{n}\frac{1}{r}\left(1-C(r)r^{\frac{2(n-k)}{k}}e^{-\frac{2n}{k^2}r^k}\right)\\
		&&-r^{k-1}\left(1+C(r)(k-1)r^{\frac{2(n-k)}{k}}e^{-\frac{2n}{k^2}r^k}
		+o\left(C(r)(k-1)r^{\frac{2(n-k)}{k}}e^{-\frac{2n}{k^2}r^k}\right)\right),
	\end{eqnarray*}
	and
		\begin{equation*}
	1+y=2-C(r)r^{\frac{2(n-k)}{k}}e^{-\frac{2n}{k^2}r^k}.\end{equation*}
Combining these expressions, we derive
	\begin{equation}\label{add}
		\frac{k}{n}\frac{C'(r)}{C(r)}=\left(3-2k+o(1)\right)C(r)r^{\frac{2(n-k)}{k}}r^{k-1}e^{-\frac{2n}{k^2}r^k}.
	\end{equation}
Thus, we can assume that there exists some sufficiently large $r_0>1$ such that  $|o(1)|<\frac{1}{2}$. Therefore, for $k = 1$ and $r > r_0$, $C(r)$ is a monotone increasing function. For $k\geq 2$ and $r\geq r_0$, $C(r)$ is a monotone decreasing function.

For any $r>r_0$, integrating \eqref{add} from $r_0$ to $r$, we obtain
	\begin{equation}\label{2.17}
		\frac{1}{C(r_0)}-\frac{1}{C(r)}=\frac{n}{k}\int_{r_0}^{r}\left(3-2k+o(1)\right)
		 s^{\frac{2(n-k)}{k}}s^{k-1}e^{-\frac{2n}{k^2}s^k}ds.
	\end{equation}
It is clear that
\begin{equation*}
	\int_{r_0}^{+\infty}r^{\frac{2(n-k)}{k}}r^{k-1}e^{-\frac{2n}{k^2}r^k}dr
\end{equation*}
is finite and positive. Thus we know that $C(r)$ has positive uniform  lower bound for $k\geq 2$. For $k=1$, by using integral by part,  we get
\begin{eqnarray*}
	&&\int_{r_0}^{+\infty}r^{2(n-1)}e^{-2nr}dr\\
	&=&\frac{1}{2n} r_0^{2(n-1)}e^{-2nr_0}+\frac{2n-2}{(2n)^2} r_0^{2n-3}e^{-2nr_0}+\frac{(2n-2)(2n-3)}{(2n)^3} r_0^{2n-4}e^{-2nr_0}\\
	&&+\cdots \frac{(2n-2)!}{(2n)^{2n-1}} e^{-2nr_0}\\
&\leq& \left(\frac{1}{2n}+\frac{2n-2}{(2n)^2}+\cdots \frac{(2n-2)!}{(2n)^{2n-1}}\right) r_0^{2(n-1)}e^{-2nr_0}.
\end{eqnarray*}
It is easy to see that, for $1\leq m\leq 2n-2$,
$$\frac{(2n-2)(2n-3)\cdots(2n-1-m)}{(2n)^{m}} \leq \left(1-\frac{1}{n}\right)^{m}.$$
Therefore, we have
\begin{eqnarray*}
	&&\int_{r_0}^{+\infty}r^{2(n-1)}e^{-2nr}dr\\
	&\leq& \frac{1}{2n}\left(1+\left(1-\frac{1}{n}\right)+\cdots +\left(1-\frac{1}{n}\right)^{2n-2}\right) r_0^{2(n-1)}e^{-2nr_0}\\
	&\leq&\frac{1}{2}\frac{1}{n-1}r_0^{2(n-1)}e^{-2nr_0}.
	\end{eqnarray*}
By \eqref{2.17}, we get
	\begin{equation}\label{2.18}
		\frac{1}{C(r_0)}-\frac{1}{C(r)}\leq \frac{n}{n-1}r_0^{2(n-1)}e^{-2nr_0}.	
		\end{equation}
	We can write that
	$$C(r_0)=\frac{1}{1-y(r_0)}r_0^{2(n-1)}e^{-2nr_0}.$$
	Since Proposition \ref{prop1} indicates that $y\rightarrow 1$ if $r\rightarrow +\infty$. Thus, if $r_0$ is sufficiently large $y_0$ is approached to 1. Hence, we may require that
\[\frac{1}{1-y(r_0)}>\frac{n}{n-1}.\]
This gives the upper bound of $C(r)$.

Consequently,  in any cases, $C(r)$ convergences to a positive constant.

	\end{proof}
	Since $y(r) = u'(r)$, we have $u(r) = u(0) + \int_{0}^{r} y(s) \, ds$, which provides the required solution. Furthermore, we can express $u(r)$ as follows:
	
	\begin{eqnarray*}
	u(r)&=&r+u(0)-C\int_0^{r}s^{\frac{2(n-k)}{k}}e^{-\frac{2n}{k^2}s^k}ds+\int^r_0o(s^{\frac{2(n-k)}{k}}e^{-\frac{2n}{k^2}s^k})ds\\
	&=&r+u(0)-C\int_0^{+\infty}s^{\frac{2(n-k)}{k}}e^{-\frac{2n}{k^2}s^k}ds+\int^{+\infty}_0o(s^{\frac{2(n-k)}{k}}e^{-\frac{2n}{k^2}s^k})ds\\
	&&+C\int_r^{+\infty}s^{\frac{2(n-k)}{k}}e^{-\frac{2n}{k^2}s^k}ds-\int^{+\infty}_{r}o(s^{\frac{2(n-k)}{k}}e^{-\frac{2n}{k^2}s^k})ds.
	\end{eqnarray*}
	Applying L'H?pital's rule, we find:
	\begin{eqnarray*}
	&&\lim_{r\rightarrow +\infty}\dfrac{\int_r^{\infty}s^{\frac{2(n-k)}{k}}e^{-\frac{2n}{k^2}s^k}ds}{r^{\frac{2n-k^2-k}{k}}e^{-\frac{2n}{k^2}r^k}}\\
	&=&\lim_{r\rightarrow +\infty}\frac{-r^{\frac{2(n-k)}{k}}e^{-\frac{2n}{k^2}r^k}}{\frac{2n-k^2-k}{k}r^{\frac{2(n-k)-k^2}{k}}e^{-\frac{2n}{k^2}r^k}-\frac{2n}{k}r^{\frac{2(n-k)}{k}}e^{-\frac{2n}{k^2}r^k}}\nonumber\\
	&=&\frac{k}{2n}\nonumber.
	\end{eqnarray*}
	Theorem \ref{thm1} can therefore be derived from Proposition \ref{prop1} and Proposition \ref{prop2}.

\section{The translating solutions}
In this section, we aim to prove Theorem \ref{es-th1}. Since the proof is standard, we will provide only a brief outline.
\par
\subsection{Constructing barriers}
We begin by constructing barrier functions for equation \eqref{teq}. Following the approach in \cite{T, RWX20}, let $z_1^k(|x|)$
denote the radial solution constructed in Section 2 with $c_0 = 0$
, whose asymptotic expansion satisfies \eqref{asy}. For $a > 0$, define $$z_a^k(|x|)=\frac{1}{a}z_1^k(a|x|).$$
It is clear that $z_a^k$ satisfies \eqref{teq} and maintains the asymptotic behavior given in \eqref{asy}.
Let $$p_i(y)=D\varphi(y)+(-1)^{i+1}2My,\,\,i=1,2$$
for any $y\in\mathbb{S}^{n-1}$.
Set
$$\tilde{z}_i^k(x,y)=\varphi(y)-p_i(y)\cdot y+z_a^k(|x+p_i(y)|),\,\forall x\in\mathbb{R}^n, y\in\mathbb{S}^{n-1}.$$
Then,
$$q_1^k(x,a)=\sup_{y\in\mathbb{S}^{n-1}}\tilde{z}_1^k(x,y)$$
is a subsolution of \eqref{teq} and
$$q_2^k(x,a)=\inf_{y\in\mathbb{S}^{n-1}}\tilde{z}_2^k(x,y)$$ is a supersolution of \eqref{teq}.
Moreover, $q_1^k(x,a)\leq q_2^k(x,a)$ and as $|x|\rightarrow+\infty$, we have
$$q_i^k(x,a)=|x|+\varphi\left(\frac{x}{|x|}\right)+O\left(|x|^{\frac{2n-k^2-k}{k}}e^{-\frac{2n}{k^2}|x|^k}\right)
,\,\,i=1, 2.$$
\par

\subsection{The approximate problem}

Let $(q_1^k)^*(\xi, a)$ be the Legendre transform of $q_1^k(x,a)$.
To ensure the convexity of our solution, we first consider the dual Dirichlet problem
on $B_\tau$ for any $0<\tau< 1$:
\be\label{D5}
\left\{
\begin{aligned}
F_*(\w\gas_{ik}\us_{kl}\gas_{lj})&=\frac{1}{a}(C_n^k)^{-1/k}\sqrt{1-|\xi|^2} \ \ \ \ \   \text{in $B_{\tau}$},\\
\us&=(q_1^k)^{*}(\xi, a) \ \ \ \ \ \ \ \ \ \ \ \ \ \ \ \  \text{on $  \p B_{\tau}$}.
\end{aligned}
\right.
\ee
Here, we have $\w=\sqrt{1-|\xi|^2},$ $\gas_{ij}=\delta_{ij}-\frac{\xi_i\xi_j}{1+\w},$ $\us_{kl}=\frac{\p^2u^*}{\p\xi_k\p\xi_l},$  $F_* (\w\gas_{ik}\us_{kl}\gas_{lj})$ $=\lt(\frac{\s_n}{\s_{n-k}}(\ka^*[\w\gas_{ik}\us_{kl}\gas_{lj}])\rt)^{1/k},$ and
$\ka^*[\w\gas_{ik}\us_{kl}\gas_{lj}]=(\ka^*_1, \cdots, \ka^*_n)$ are the eigenvalues of the matrix $(\w\gas_{ik}\us_{kl}\gas_{lj}).$
The solvability of \eqref{D5} has been established in Section 3.1 and Section 7.2 in \cite{RWX20}.
Suppose $u^*_{\tau}$ is the solution of \eqref{D5}. Let $u_{\tau}$ be the Legendre transform of $u_{\tau}^*$.
Therefore, by standard PDE theorems, to prove
Theorem \ref{es-th1},  we only need to obtain
local  estimates   \eqref{teq}.

\par
\subsection{Local estimates}
Similar to the proof of Section 7.3 in \cite{RWX20},
we have the following local estimates for translating solitons.

\begin{lemma}\label{lem7.2.1} Let $u^*_{\tau}$ be a solution to equation \eqref{D5} and $u_\tau$ be the Legendre transform of $u^*_{\tau}.$ Then,
for any $x\in Du^*_{\tau}(B_\tau),$ we have
 $$q_1^k(x,a)\leq u_\tau(x)\leq q_2^k(x,a).$$
\end{lemma}

\begin{lemma}
\label{lemm81}
Let $\Omega\subset \R^n$ be a bounded open set. Let $u, \bar{u}, \Psi:\Omega\rightarrow\mathbb{R}^n$ be strictly spacelike, i.e.
$$|Du|,|D\uu|,|D\Psi|<1.$$
Assume that $u$ is strictly convex and $u\leq\bar{u}$ in $\Omega.$ Also
assume that near $\partial\Omega,$ we have $\Psi>\bar{u}.$
Consider the set where $u>\Psi.$ For every $x$ in that set,  the following gradient estimate for $u$ holds:
\[\frac{1}{\sqrt{1-|Du|^2}}\leq\frac{1}{u(x)-\Psi(x)}\cdot\sup\limits_{\{u>\Psi\}}\frac{\bar{u}-\Psi}{\sqrt{1-|D\psi|^2}}.\]
\end{lemma}

For any compact convex domain $\K$,  let $2\delta=\min_{\K}(q_1^k(x,a)-q_1^k(x,a+100))$. Define
$$\Psi(x)=q_1^k(x, a+100)+\delta.$$
It is clear that for sufficiently large $|x|$,   we have $\Psi(x)>q_2(x)$.
On the other hand, we have $\Psi(x)<q_1(x)$ in $\K$.
By applying Lemma \ref{lemm81}, we obtain the local $C^1$ estimate.
Moreover, we can  establish the following Pogorelov type local $C^2$ estimate
for translating solitons.
\begin{lemma}
\label{lc2lem1}
Let $u$ be the solution of \eqref{teq} defined on $\Omega$.
For any given $s>\min\limits_{\R^n}u(x)+1,$ assume that $u|_{\T\Omega}>s.$ Let $\kappa_{\max}(x)$
 denote the largest principal curvature of $\M_{u}=\{(x, u(x))|x\in\Omega\}$ at $x$.
Then, we have
\[\max\limits_{\M_{u}}(s-u)\kappa_{\max}\leq C_1.\]
Here, $C_1$ only depends on the local $C^1$ estimate of $u$. More specifically, $C_1$ depends on the upper bound of
$-\langle \nu, E\rangle.$
\end{lemma}

%\subsection{The incompleteness}

\section{The curvature flow}
In this section, we will prove Theorem \ref{theo1}.  Since our proof follows the paper \cite{WX22-2}, we will only give the differences in the argument. Therefore,  the proof presented here is a sketch and one may find the detail in \cite{WX22-2}.

Io establish the solvability and convergence of the flow, we will examine the solvability of the equation satisfied by $u^*$, which is the Legendre transform of
$u$. Specifically, we will focus on the approximate problem in $B_r$, $r<1$.
\begin{eqnarray}\label{ubr2}
	\left\{
	\begin{aligned}
		 (u_r^*)_t&=-(C_n^k)^{-1/k}F_*^{-1}(w^*\gamma^*_{ik}u^*_{kl}\gamma^*_{lj})w^*\ \ \   \mathrm{in}\  B_r\times(0,T]\\
	 u_r^*&=u^*_0-at \  \ \ \ \ \ \ \ \ \ \ \ \ \ \  \ \ \ \ \ \ \ \ \ \ \ \ \ \ \ \ \  \mathrm{on}\  \partial B_r\times(0,T],\\
		u_r^*(0)&=u^*_0\ \ \ \ \ \ \ \ \ \ \ \ \ \ \ \ \ \ \ \ \ \ \ \ \ \ \ \ \ \ \ \ \ \ \ \ \ \ \  \mathrm{in}\ B_r\times {0}
	\end{aligned}
	\right.
\end{eqnarray}
where \[F_*=\left(\frac{\sigma_n}{\sigma_{n-k}}\right)^{\frac{1}{k}}(w^*\gamma^*_{ik}u^*_{kl}\gamma^*_{lj}), \w=\sqrt{1-|\xi|^2}, \gas_{ij}=\delta_{ij}-\frac{\xi_i\xi_j}{1+\w}, \us_{kl}=\frac{\p^2u^*}{\p\xi_k\p\xi_l}.\]

\subsection{$C^0$ estimates}
By scaling, the assumption \eqref{assume} becomes
$$0<\sigma_k^{\frac{\alpha}{k}}(\kappa[M_{u_0}])<-a(C_n^k)^{1/k}\langle\nu_{u_0}, E\rangle.$$
Therefore, $u_0$ is the sup solution of \eqref{teq}.
From the previous section, we know that \eqref{teq} has a smooth solution $\underline{u}$. Then we have $\underline{u}\leq u_0$.
Moreover, both $u_0$ and $\underline{u}$ satisfy the  asymptotic behavior
$$\underline{u},u_0 \rightarrow |x|+\varphi\left(\frac{x}{|x|}\right),\ \text{as} \ |x|\rightarrow \infty.$$

Then,  $\underline{u}+at, u_0+at$ are the sub solution and sup solution of \eqref{ge}, respectively. Therefore, we have
$\underline{u}^*-at, u_0^*-at$ are the sup solution and sub solution of \eqref{ubr2}, where $\underline{u}^*,\bar{u}^*$ are the Legendre transform of $\underline{u},u_0$.  Thus, we have
$$u_0^*-at\leq u^*_r\leq \underline{u}^*-at.$$

\subsection{The  bound of $F_*$}
Consider the hyperplane $\mathbb{P}:=\{X=(x_1, \cdots, x_{n}, x_{n+1}) |\, x_{n+1}=1\}$ and  the projection of
$\mathbb{H}^n(-1)$ from the origin into $\mathbb{P}.$ Then $\mathbb{H}^n(-1)$ is mapped in
a one-to-one fashion onto an open unit ball $B_1:=\{\xi\in\R^n |\, \sum\xi^2_k<1\}.$ The mapping
$P$ is given by
\[P: \mathbb{H}^n(-1)\rightarrow B_1;\,\,(x_1, \cdots, x_{n+1})\mapsto (\xi_1, \cdots, \xi_n),\]
where $x_{n+1}=\sqrt{1+x_1^2+\cdots+x_n^2},$ $\xi_i=\frac{x_i}{x_{n+1}}.$ Set $U_r:=P^{-1}(B_r)\times(0, T]$.
For $u^{*}_r$ satisfying \eqref{ubr2}, we define $v_r=\frac{u^{*}_r}{w^*}$. Then a straight forward calculation shows that $v_r$ satisfies
\be\label{sap3.1}
\left\{
\begin{aligned}
(v_r)_t&=-(C_n^k)^{-1/k}F_*^{-1}(\Lambda_{ij}):=\tilde{G}(\Lambda_{ij})\,\,&\mbox{in $U_r\times(0, T]$}\\
v_r(\cdot, t)&=\frac{u_0^*-at}{\sqrt{1-r^2}}\,\,&\mbox{on $\p U_r\times[0, T]$}\\
v_r(\cdot, 0)&=u_0^*x_{n+1}\,\,&\mbox{on $U_r\times\{0\},$}
\end{aligned}
\right.
\ee
where $\Lambda_{ij}=\bn_{ij}v_r-v_r\delta_{ij},$ and $\bn$ is the Levi-Civita Connection of the hyperbolic space.

\begin{lemma}
\label{sap-lem3.1}
Assume  $v$ is a solution of \eqref{sap3.1}. Then we have $$\frac{1}{C_3x_{n+1}}>F_*>\frac{1}{C_2 x_{n+1}}\,\, \mbox{on $\bar{U}_r\times(0, T]$.}$$
Here,  $C_2$ is a uniform constant and $C_3$ is a constant depending on the lower bound of $\sigma_k^{\frac{1}{k}}(\ka[\M_0])$ on $\bar{U}_r$.
\end{lemma}
\begin{proof}
Since
\[\tilde{G}_t=\tilde{G}^{ij}((v_t)_{ij}-v_t\delta_{ij})=\tilde{G}^{ij}(\bn_{ij}\tilde{G}-\tilde{G}\delta_{ij}),\]
we have $\mathcal{L}\tilde{G}=-\tilde{G}\sum_{i}\tilde{G^{ii}},$ where $\mathcal{L}:=\frac{\p}{\p t}-\tilde{G}^{ij}\bn_{ij}.$
It is clear that $\mathcal{L}x_{n+1}=-x_{n+1}\sum_i\tilde{G}^{ii}$.
Therefore, we get
$$\mathcal{L}\frac{-\tilde{G}}{x_{n+1}}=2\tilde{G}^{ij}\frac{(x_{n+1})_i}{x_{n+1}}\left(\frac{-\tilde{G}}{x_{n+1}}\right)_j.$$
Applying the maximum principal, we find  that $F_*x_{n+1}$ achieves its maximum and minimum at the parabolic boundary.
In view of the short time existence theorem, we obtain on $\p U_r\times (0, T]$:
\[(C^k_n)^{1/k}F_*=\frac{\sqrt{1-r^2}}{a}.\]
Therefore, on $\p U_r\times(0,T]$, we get  $$(C^k_n)^{1/k}F_*x_{n+1}=\frac{1}{a}.$$
On $U_r\times\{0\}$, we have
$$c_0x^{-1}_{n+1}<F_*^{-1}x_{n+1}^{-1}<C.$$
This concludes the argument.
\end{proof}

\subsection{$C^1$ estimates}

\begin{lemma}
\label{sap-lem4.1}
Let $u^*_r$ be a solution of \eqref{ubr2}. Then
$|Du_r^*|\leq C$, where $C=C(\M_0, r, t, T)$ is a constant in $\bar{B}_r \times[0, T]$.
\end{lemma}

 The proof is similar to \cite{WX22-2}. We only need to change the given boundary function to be $u_0^*-at$.

\subsection{$C^2$ estimates}

From now on, for our convenience, we will consider the solution of equation \eqref{sap3.1}. Assume $\{\tau_1,\cdots,\tau_n\}$ is the orthonormal frame of the boundary $\p U_r$, where $\tau_1,\cdots,\tau_{n-1}$ are the tangential vectors and $\tau_n$ is the unit interior normal vector.
 \begin{lemma}
\label{sap-lem5.1}
Let $v$ be the solution of \eqref{sap3.1}. Then the second order tangential derivatives on the boundary satisfy
$|\bn_{\al\beta} v|\leq C$ on $\p U_r\times (0, T]$ for $\al, \beta<n.$ Here $C$ depends on $\M_0, r,T$.
\end{lemma}
The proof follows the same approach as in \cite{WX22-2}, except that the subsolution is defined by
\begin{eqnarray}\label{eq2}
\ubar{v}=\frac{u_0^*-at}{\sqrt{1-|\xi|^2}}.
\end{eqnarray}

Now, we will show that $|\bn_{\al n}v|$ is bounded. In the following we  denote the operator $\mathfrak{L}$ by
\be\label{operator-L}
\mathfrak{L}\phi:=\phi_t-\tilde{F}_v^{-2}\tilde{F}_v^{ij}\bn_{ij}\phi+\phi\tilde{F}_v^{-2}\sum_i\tilde{F}^{ii}_v
\ee
for any smooth function $\phi.$ Here, $\tilde{F}_v(\Lambda_{ij})=F_*(\Lambda_{ij}),$ $\tilde{F}^{ij}_v=\frac{\p\tilde{F}_v}{\p\Lambda_{ij}},$
and $\Lambda_{ij}=\bn_{ij}v-v\delta_{ij}.$

\begin{lemma}
\label{sap-lem5.2}
Let $v$ be a solution of \eqref{sap3.1}, and let $\ubar{v}$ be the subsolution of \eqref{sap3.1} defined by \eqref{eq2}.
Define
$h=(v-\ubar{v})+B\lt(\frac{1}{\sqrt{1-r^2}}-x_{n+1}\rt),$ where $B>0$ is a constant. Then for any given constant $B_1>0$, there exists a sufficiently large $B$ depending on $\M_0, r, T$, such that $\mathfrak{L}h>\frac{B_1}{\tilde{F}_v^2}\sum\tilde{F}^{ii}_v.$
\end{lemma}

The proof is also similar to that in \cite{WX22-2}, with the main difference being the verification of the last step. Since
\[\tilde{F}(\ubar{\Lambda}_{ij})=F_*(w^*\gas_{ik}\tbus_{kl}\gas_{lj})< c_0^{-1},\]
where $c_0$ is the lower bound of $\sigma_k^{\frac{1}{k}}(\ka[\M_0])$ on $\bar{U}_r$.
We divide  into two cases: $\tilde{F}_v>c_0^{-1}(1+\al),$  and   $\tilde{F}_v\leq c_0^{-1}(1+\al),$ to discuss.

\begin{lemma}
\label{sap-lem5.3}
Let $v$ be a solution of \eqref{sap3.1} and suppose $\tau_n$ is the interior unit normal vector filed of $\p U_r.$ We have $|\bn_{\al n}v|\leq C$
on $\p U_r\times (0, T]$. Here $C$ depends on $\M_0, r, T$.
\end{lemma}
 We only need to check that
\[
\begin{aligned}
|\mathfrak{L}\mathcal{T}\ubar{v}|
=\bigg|-\tilde{G}^{ij}\bn_{ij}\left(\mathcal{T} \frac{u_0^*}{w^*}\right)+\mathcal{T} \frac{u_0^*}{w^*}\sum_i\tilde{G}^{ii}\bigg|
\leq C_4\sum_i\tilde{G}^{ii},
\end{aligned}
\]
where $\mathcal{T}=\xi_{\al}\p_n-\xi_n\p_{\al}$. The rest proof is same as \cite{WX22-2}.

 \begin{lemma}\label{sap-lem5.4}
 Let $v$ be a solution of \eqref{sap3.1} and suppose $\tau_n$ is the interior unit normal vector filed of $\p U_r.$ Then we have $|\bn_{nn}v|\leq C$
on $\p U_r\times (0, T],$ where  $C$ is a positive constant depending on $\M_0, r, T$.
 \end{lemma}
The proof is similar to \cite{WX22-2}. We only need to set $\al=1$ and replace the function
$$\dfrac{[2\tilde{t}]^{1/2}\sqrt{1-r^2}}{-u_0^*}$$ with the function $$\dfrac{\sqrt{1-r^2}}{(C_n^k)^{1/k}a}.$$

The global $C^2$ estimates for $u_r^*$ directly follows from Lemma 20 in \cite{WX212}. Therefore, the solvability of the approximate problem \eqref{ubr2} is established.

%\section{Local estimates and convergence}
\subsection{Local $C^0$ estimates}
\label{lc2}
Suppose $u_r$ is the Legendre transform of  $u^*_r$.
By Lemma 13 in \cite{WX20}, we get
\[\ubar{u}\lt(x\rt) +at
<u_r(x,t)<u_0\lt( x\rt)+at.\]

\subsection{Local $C^1$ estimates}
\label{sub-loc-c1}
We construct a new subsolution $\lu_1$ satisfying
\[\s_k^{\frac{1}{k}}(\ka[\M_{\lu_1}(x)])=-100(C_n^k)^{1/k}a\lt<\nu_{\lu_1}, E\rt>\]
and as $|x|\goto\infty$
$$\lu_1\goto |x|+\varphi\left(\frac{x}{|x|}\right).$$
By the strong maximum principle we have
$\lu_1(x)<\lu(x)$. Then a similar argument as \cite{WX22-2} gives the local $C^1$ estimate.

\subsection{Local estimates for $F$}
Denote  $v=-\lt<\nu, E\rt>$.  Recalling Lemma \ref{sap-lem3.1}, we have $\Phi=\sigma_k^{1/k}(\kappa[\M_{u_r}])< C_2 v.$ A same proof in \cite{WX22-2} gives

\begin{lemma}
\label{loc-F-lem}
Let $u_r^*$ be the solution of \eqref{ubr2} and $u_r$ be the Legendre transform of $u_r^*.$ For any $c>0,$ denote $K:=\{(x, t)\mid u_r(x, t)+t\leq c\}$
and $V_0:=\max\limits_{(x, t)\in K}v.$ Then we have
\[\lt(\frac{c-t-u_r}{c}\rt)^\gamma e^{2(v-V_0)}\frac{1}{(\al-1)V_0}\leq\Phi<C_2V_0,\]
where $C_2$ is the uniform constant determined by Lemma \ref{sap-lem3.1} and $\gamma=4+8V_0^2.$ Note that $c$ is always chosen such that
$K\subset\bigcup_{t\in[0, \infty)} \lt(Du_r^*(B_r, t)\times\{t\}\rt).$
\end{lemma}

\subsection{Local $C^2$ estimates}
As in \cite{WX22-2}, we can prove the following:
\begin{lemma}
\label{loc-c2-lem}
Let $u_r^*$ be the solution of \eqref{ubr2} and $u_r$ be the Legendre transform of $u_r^*.$ Denote $\Omega_r(t):=Du_r^*(B_r, t).$ For any given $c>0,$ let $r_c\in (0, 1)$ such that when $r>r_c,$ $u_r(\cdot, t)|_{\p\Omega_r(t)}>c$ for all $t\in[0, \infty).$ Then for $r>r_c$ we have
\[(c-u_r)^m\log\ka_{\max}(x, t)\leq C,\]
where $\ka_{\max}(x, t)$ is the largest principal curvature of $\M_{u_r}$ at $(x, t),$
$m$ is a large constant depending only on $k,$ and $C:=C(\Phi, c)>0$ is independent of $r.$
\end{lemma}

\subsection{Covergence}	
Now, denote \be\label{urs}\td{u}^*_r(x,t)=u_r^*(x,t)+at,\ee
then $\td{u}^*_r$ satisfies
\be\label{conv1.2}
\left\{
\begin{aligned}
(\td{u}^*_r)_{t}&=-(C_n^k)^{-1/k}F_*^{-1}(w^*\gas_{ik}(\td{u}^*_r)_{kl}\gas_{lj})w^*+a\,\,&\mbox{in $B_r\times(0, T]$},\\
\td{u}^*_r(\cdot, t)&=u_0^*\,\, &\mbox{on $\p B_r\times[0, T],$}\\
\td{u}^*_r(\cdot, 0)&=u^*_0\,\,&\mbox{on $B_r\times\{0\}.$}
\end{aligned}
\right.
\ee
As in \cite{WX22-2}, we can establish the following two lemmas.
\begin{lemma}
\label{conv-lem1}
Let $\td{u}_r^*$ be defined as in  \eqref{urs}. Then we have $\td{u}^*_r(\cdot, t)\goto u_r^{\infty*}(\cdot)$ uniformly in $B_r$ as $t\goto\infty.$ Here $u^{\infty*}_r$
satisfies
\be\label{conv1.3}
\left\{
\begin{aligned}
F_*^{-1}(w^*\ga^*_{ik}(u^{\infty*}_r)_{kl}\gamma_{lj}^*)w^*&=a(C_n^k)^{1/k}\,\,&\mbox{in $B_r$},\\
u^{\infty*}_r&=u^*_0\,\,\ \ \ &\mbox{on $\p B_r$}.
\end{aligned}
\right.
\ee
\end{lemma}

\begin{lemma}
\label{conv-lem2}
Let $u_r$ be the Lengendre transform of $u_r^*$. Then $u_r(x, t)\goto u(x, t)$ uniformly in any compact subset of
$\R^n\times[0, \infty)$ as $r\goto 1$.
\end{lemma}

From Section 3 of \cite{WX22-1}, we observe that as $r \to 1$, $u_r^{\infty}$, which is the Legendre transform of $u_r^{\infty*}$, converges uniformly to $u^\infty$on any compact set $K \subset \mathbb{R}^n$. Moreover, $u^\infty$
satisfies

\[
\left\{
\begin{aligned}
\s_k^{\frac{1}{k}}(\ka[\M_{u^\infty}])&=-a(C_n^k)^{1/k}\lt<\nu_{u^\infty}, E\rt>,\\
u^\infty-|x|&\goto\varphi\lt(\frac{x}{|x|}\rt).
\end{aligned}
\right.
\]
We have completed the proof of Theorem \ref{theo1}.

\end{document}